\documentclass[12pt]{amsart}

\usepackage[english]{babel}
\usepackage{times,amsfonts,amsmath,amssymb} 
\usepackage{pdfsync}
\usepackage[notref,notcite]{showkeys}
\usepackage{graphicx} 

\usepackage{color}

\definecolor{gr}{rgb}   {0.,   0.69,   0.23 }
\definecolor{bl}{rgb}   {0.,   0.5,   1. }
\definecolor{mg}{rgb}   {0.85,  0.,    0.85}
\definecolor{or}{rgb}   {0.9,  0.5,   0.}

\definecolor{webred}{rgb}{0.75,0,0}
\definecolor{webgreen}{rgb}{0,0.75,0}
\usepackage[citecolor=webgreen,colorlinks=true,linkcolor=webred]{hyperref}

\makeatletter
\def\paragraph{\@startsection{paragraph}{4}%
  \z@{0.3em}{-.5em}%
  {$\bullet$ \ \normalfont\itshape}}
\makeatother

\newtheorem{theorem}{Theorem}[section]
\newtheorem{proposition}[theorem]{Proposition}
\newtheorem{lemma}[theorem]{Lemma}
\newtheorem{corollary}[theorem]{Corollary}

\theoremstyle{definition}

\theoremstyle{remark}
\newtheorem{remark}[theorem]{Remark}

\newcommand{\Bk}{\color{black}}

\newcommand{\N}{\mathbb{N}}
\newcommand{\R}{\mathbb{R}}
\newcommand{\C}{\mathbb{C}}

\newcommand{\cB}{\mathcal{B}}
\newcommand{\cC}{\mathcal{C}}

\newcommand{\cF}{\mathcal{F}}

\newcommand{\cT}{\mathcal{T}}

\newcommand{\cK}{\mathcal{K}}
\newcommand{\cX}{\mathcal{X}}

\newcommand{\gh}{\mathfrak{h}}

\newcommand{\gJ}{\mathfrak{J}}

\newcommand{\bel}{\begin{equation} \label}
\newcommand{\ee}{\end{equation}}
\newcommand{\bea}{\begin{eqnarray}}
\newcommand{\eea}{\end{eqnarray}}
\newcommand{\beas}{\begin{eqnarray*}}
\newcommand{\eeas}{\end{eqnarray*}}

\newcommand{\pd}{\partial}

\newcommand{\rd}{\mathrm{d}}

\newcommand\curl{\operatorname{curl}}

\renewcommand\Re{\operatorname{Re}}

\title[Limit absorption for the half-plane magnetic Dirichlet Laplacian]{Limiting absorption principle for the magnetic Dirichlet Laplacian in a half-plane.}

\author[N.\ Popoff]{Nicolas Popoff}
\address{Universit\'e de Bordeaux, IMB, UMR 5251, 33405 TALENCE cedex, France}
\email{Nicolas.Popoff@math.u-bordeaux1.fr}

\author[E.\ Soccorsi]{Eric Soccorsi}
\address{Aix Marseille Universit\'e, Universit\'e de Toulon, CNRS, CPT UMR 7332, 13288, Marseille, France}
\email{eric.soccorsi@univ-amu.fr}

\begin{document}

\begin{abstract}
We consider the Dirichlet Laplacian in the half-plane with constant magnetic field. Due to the translational invariance this operator admits a fiber decomposition and a family of dispersion curves, that are real analytic functions. Each of them is simple and monotonically decreasing from positive infinity to a finite value, which is the corresponding Landau level. These finite limits are thresholds in the purely absolutely continuous spectrum of the magnetic Laplacian. We prove a limiting absorption principle for this operator both outside and at the thresholds. Finally, we establish analytic and decay properties for functions lying in the absorption spaces. We point out that the analysis carried out in this paper is rather general and can be adapted to a wide class of fibered magnetic Laplacians with thresholds in their spectrum that are finite limits of their band functions..
\end{abstract}

\maketitle \thispagestyle{empty}




\section{Introduction}

In the present article we consider the Hamiltonian $(-i \nabla + A)^2$ with magnetic potential $A(x,y):=(0,x)$, defined in the half-plane $\Omega:=\{(x,y)\in \R^2, \, x>0\}$. We impose Dirichlet boundary conditions at $x=0$ and introduce the self-adjoint realization
$$ H := -\pd_x^2 + (-i \pd_y - x)^2, $$
initially defined in $C_0^{\infty}(\Omega)$ and then closed in $L^2(\Omega)$. This operator models the planar motion of a quantum charged particle (the various physical constants are taken equal to $1$) constrained to $\Omega$ and submitted to an orthogonal magnetic field of strength $\curl A=1$, it has already been studied in several articles (e.g., \cite{DeBPu99, HeMo01, BruMirRai13, HisPofSoc14, Ivrii}).

The Schr\"odinger operator $H$ is translationally invariant in the $y$-direction and admits a fiber decomposition with fiber operators which have purely discrete spectrum. The corresponding dispersion curves (also named band functions in this text)
are real analytic functions in $\R$, monotonically decreasing from positive infinity to the $n$-th Landau level $E_n:=2n-1$ for $n \in \N^*$.
As a consequence, the spectrum of $H$ is absolutely continuous, 
equals the interval $[E_1,+\infty)$. Hence the resolvent operator $R(z):=(H-z)^{-1} \in \cB(L^2(\Omega))$ depends analytically on $z$ in $\C \setminus [E_1,+\infty)$, and $R^\pm(z):=R(z)$ is well defined for every $z \in \C^\pm:=\{ z \in \C,\ \pm \Im z >0\}$. 

Since $H$ has a continuous spectrum, the spectral projector $E(a,b)$ of $H$, associated with the interval $(a,b)$, $a<b$, expresses as
$$ E(a,b) = \frac{1}{2 i \pi}  \lim_{\varepsilon \downarrow 0} \int_a^b \left( R(\lambda + i \varepsilon) - R(\lambda - i \varepsilon) \right) \rd \lambda,$$
by the spectral theorem. Suitable functions of the operator $H$ may therefore be expressed in terms of the limits of the resolvent operators $\lim_{\varepsilon \downarrow 0} R(\lambda \pm i \varepsilon)$ for $\lambda \in \R$. As a matter of fact the Schr\"odinger propagator $e^{-i t H}$ associated with $H$ reads
$$ e^{-it H} = \frac{1}{2 i \pi} \lim_{\varepsilon \downarrow 0} \int_{E_1}^{+\infty} e^{-it \lambda} \left( R(\lambda + i \varepsilon) - R(\lambda - i \varepsilon) \right) \rd \lambda,\ t >0.$$
This motivates for a quantitative version of the convergence $\lim_{\varepsilon \downarrow 0} R(\lambda \pm i \varepsilon)$, for $\lambda \in [E_1,+\infty)$, known as the {\it limiting absorption principle} (abbreviated to LAP in the sequel). Notice moreover that a LAP is a useful tool for the analysis of the scattering properties of $H$, and more specifically for the proof of the existence and the completeness of the wave operators (see e.g. \cite[Chap. XI]{RS3}). The main purpose of this article is to establish a LAP for $H$. That is, for each $\lambda \in [E_1,+\infty)$, we aim to prove that $R(\lambda \pm i \varepsilon)$ has a limit as $\varepsilon \downarrow 0$, in a suitable sense we shall make precise further. 

There is actually a wide mathematical literature on LAP available for various operators of mathematical physics (see e.g. \cite{Va66, Ei69, IkeSai72, Ag75, Ta81a, Ta81b, BenADe87, DeGui88}). 
More specifically, the case of analytically fibered self-adjoint operators was addressed in e.g. \cite{CroDe95, GeNi98, So01}. Such an operator $A$ is unitarily equivalent to the multiplier by a family of real analytic dispersion curves, so its spectrum is the closure of the range of its band functions. Generically, energies associated with a ``flat" of any of the band functions $\{ \lambda_n,\ n \in \N^* \}$, are {\it thresholds} in the spectrum of $A$. More precisely, a threshold of the operator $A$ is any real number $\lambda$ satisfying
$\inf_{U} |\lambda_n' \circ \lambda_n^{-1}| =0$ for some $n \in \N^*$ and all neighborhoods $U$ of $\lambda$ in $\overline{\lambda_n(\R)}$. We call $\cT$ the set of thresholds.

The occurrence of a LAP outside the thresholds of analytically fibered operators is a rather standard result. It is tied to the existence of a Mourre inequality at the prescribed energies (see \cite{GeNi98, Ge08}), arising from the non-zero velocity of the dispersion curves for the corresponding frequencies.
More precisely, given an arbitrary compact subset $K \subset \C \setminus \cT$, we shall extend $z \mapsto R^\pm(z)$ to a H\"older continuous function on $K \cap \overline{\C^\pm}$ in the norm-topology of $\cB(L^{2,s}(\Omega),L^{2,-s}(\Omega))$ for any $s \in (1 \slash 2,+\infty)$. Here and henceforth the Hilbert space
$$L^{2,\sigma}(\Omega):=\{ u : \Omega \to \C\ \mbox{measurable},\ (x,y) \mapsto (1+y^2)^{\sigma \slash 2} u(x,y) \in L^2(\Omega) \}, $$ 
is endowed with the scalar product $\langle u,v \rangle_{L^{2,\sigma}(\Omega)}:=\int_{\Omega} (1+y^2)^\sigma u(x,y) \overline{v(x,y)} \rd x \rd y$.

Evidently, local extrema of the dispersion curves are thresholds in the spectrum of fibered operators. Any such energy being a critical point of some band function, it is referred as an {\it attained threshold}. Actually, numerous operators of mathematical physics modeling the propagation of acoustic, elastic or electromagnetic waves in 
stratified media \cite{CroDe95, BoDe99, Bo00, So01} and various magnetic Hamiltonians \cite{GeSe97, HeMo01, Yaf08} have all their thresholds among local minima of their band functions. A LAP at an attained threshold may be obtained upon imposing suitable vanishing condition (depending on the level of degeneracy of the critical point) on the Fourier transform of the functions in $L^{2,\sigma}(\Omega)$, at the corresponding frequency. See \cite{BoDe99,So01} for the analysis of this problem in the general case. 

Nevertheless, none of the above mentioned papers seems relevant for our operator $H$. This comes from the unusual behavior of the band functions of $H$ at infinity: In the framework examined in this paper, there exists a countable set $\cT$ of thresholds $\{ E_n,\ n \in \N^* \}$ in the spectrum of $H$, but in contrast with the situations examined in \cite{BoDe99, So01}, none of these thresholds are attained. This peculiar behavior raises several technical problems in the derivation of a LAP for $H$ at $E_n$, $n \in \N^*$. Nevertheless, for any arbitrary compact subset $K \subset \C$ (which may contain one or several thresholds $E_n$ for $n \in \N^*$) we shall establish in Theorem \ref{T:thresholds} a LAP for $H$ in $K$ for the topology of the norm in $\cB(X_K,(X_K)')$, where $X_K$ is a suitable subspace of $L^{2,\sigma}(\Omega)$, for an appropriate $\sigma > 1 \slash 2$, which is dense in $L^2(\Omega)$. The space $X_{K}$ is made of $L^{2,\sigma}(\Omega)$-functions, with smooth Fourier coefficients vanishing suitably at the thresholds of $H$ lying in $K$. Otherwise stated there is an actual LAP at $E_n$, $n \in \N^*$, even though $E_n$ is a {\it non attained threshold} of $H$. 
Moreover, it turns out that the method developed in the derivation of a LAP for $H$ is quite general and may be generalized to a wide class of fibered operators (such as the ones examined in \cite{HeMo01, Yaf08, BruMirRai13, HiSoc1, Pof13IV}) with non attained thresholds in their spectrum.

Finally, functions in $X_{K}$ exhibit interesting geometrical properties. Namely, assuming that $E_n \in K$, it turns out that the asymptotic behavior of the $n$-th band function of $H$ at positive infinity (computed in \cite[Theorem 1.4]{HisPofSoc14}) translates into
super-exponential decay in the $x$-variable (orthogonal to the edge) of their $n$-th harmonic, see Theorem \ref{thm1}. 
Such a behavior is typical of magnetic Laplacians, as explained in Remark \ref{R:symbole}.

\subsection{Spectral decomposition associated with the model}
Let us now collect some useful information on the fiber decomposition of the operator $H$.

The Schr\"odinger operator $H$ is translationally invariant in the longitudinal direction $y$ and therefore allows a direct integral decomposition
\bel{E:integraldirect}
\cF_y^{*}H\cF_y= \int_{k\in \R}^{\bigoplus} \gh(k) \rd k,
\ee
where $\cF_y$ denotes the partial Fourier transform with respect to $y$ and the fiber operator $\gh(k):=-\partial_{x}^2+(x-k)^2$ acts in $L^2(\R_{+}^*)$ with a Dirichlet boundary condition at $x=0$. Since the effective potential $(x-k)^2$ is unbounded as $|x|$ goes to infinity, each $\gh(k)$, $k \in \R$, has a compact resolvent, hence a purely discrete spectrum. We note $\{ \lambda_n(k),\ n \in \N^* \}$ the non-decreasing sequence of the eigenvalues of $\gh(k)$, each of them being simple. Furthermore, for $k \in \R$ introduce a family $\{ u_n(\cdot,k),\ n \in \N^* \}$ of eigenfunctions of the operator $\gh(k)$, which satisfy 
$$ \gh(k) u_n(x,k) = \lambda_n(k) u_n(x,k),\ x \in \R_+^*, $$
and form an orthonormal basis in $L^2(\R_+^*)$.

As $\{ \gh(k),\ k \in \R \}$ is a Kato analytic family, the functions $\R \ni k \mapsto \lambda_n(k) \in (0,+\infty)$, $n \in \N^*$, are analytic (see e.g. \cite[Theorem XII.12]{ReSi78}). Moreover they are monotonically decreasing in $\R$ according to \cite[Lemma 2.1 (ii)]{DeBPu99} and the Max-Min principle yields (see \cite[Lemma 2.1 (iii) and (v)]{DeBPu99})
$$ \lim_{k \to -\infty} \lambda_n(k)=+\infty\ \mbox{and}\ \lim_{k \to +\infty}  \lambda_n(k) =E_n,\ n \in \N^*. $$
Therefore, the general theory of fibered operators (see e.g. \cite[Section XIII.16]{ReSi78}) implies that the spectrum of $H$ is purely absolutely continuous, with
$$ \sigma(H) = \overline{\cup_{n \in \N^*} \lambda_n(\R)}=[E_1,+\infty]. $$

For all $n \in \N^*$ we define the  $n^{\textrm{th}}$ Fourier coefficient $f_n \in L^2(\R)$ of $f \in L^2(\Omega)$ by 
$$ f_{n}(k):=\langle \cF_{y} f (\cdot,k),u_{n}(\cdot,k) \rangle_{L^2(\R_{+})},\ k \in \R, $$
and introduce its $n^{\textrm{th}}$ harmonic as
\bel{D:Pn}
\Pi_{n}f(x,y):= \int_{\R} e^{iky} f_{n}(k) u_{n}(x,k) \rd k,\ (x,y) \in \Omega.
\ee
In view of \eqref{E:integraldirect}, we have the standard Fourier decomposition $f=\sum_{n\in \N^{*}}\Pi_{n}f$ in $L^2(\Omega)$ and the following Parseval identity
\bel{s1}  
\|f\|^2_{L^2(\Omega)}=\sum_{n\in \N^{*}} \| f_{n} \|_{L^2(\R)}^2,
\ee
involving that the linear mapping $\pi_{n} : f \mapsto f_{n}$ is continuous from $L^2(\Omega)$ into $L^2(\R)$ for each $n \in \N^*$. Let us now recall the following useful properties of the restriction of $\pi_n$ to $L^{2,s}(\Omega)$ for $\in (1 \slash 2,+\infty)$
(see e.g. \cite[Proposition 3.2]{CroDe95}). 
\begin{lemma}
\label{L:fnholder}
Fix $s>1 \slash 2$. Then the operators $\pi_n$ are uniformly bounded with respect to $n \in \N^*$ from $L^{2,s}(\Omega)$ into $L^{\infty}(\R)$: 
\bel{s2}
\exists C(s)>0,\ \forall n \in \N^*,\ \forall f \in L^{2,s}(\Omega),\ \forall k \in \R,\ | f_n(k) | \leq C(s) \| f \|_{L^{2,s}(\Omega)}.
\ee
Moreover for any $\alpha \in [0, s- 1 \slash 2)$, each operator $\pi_n$, $n \in \N^*$, is bounded from $L^{2,s}(\Omega)$ into $\cC_{\rm loc}^{0,\alpha}(\R)$, the set of locally H\"older continuous functions in $\R$, of exponent $\alpha$. Namely, there exists a function $C_{n,\alpha,s} \in \cC^0(\R^2;\R_+)$ such that,
\bel{s3}
\forall f \in L^{2,s}(\Omega),\ \forall (k,k') \in \R^2,\ | f_n(k') - f_n(k) | \leq C_{n,\alpha,s}(k,k') \| f \|_{L^{2,s}(\Omega)} | k'-k|^{\alpha}.
\ee
\end{lemma}

\section{Limiting absorption principle}
For all $z \in \C \setminus [E_1,+\infty)$ and $f, g \in L^2(\R_{+}^2)$, standard functional calculus yields
\bel{l1}
\langle R(z) f , g\rangle_{L^2(\R_{+}^2)}=\sum_{n \geq 1} r_{n}(z)\ \ \mbox{with}\ \ r_{n}(z):=\int_{\R}\frac{f_{n}(k)\overline{g_{n}(k)}}{\lambda_{n}(k)-z} \rd k,\ n \in \N^*.
\ee
Since $I_n:=\lambda_{n}(\R)=(E_n,+\infty)$ for each $n \in \N^*$, the function $z \mapsto r_{n}(z)$ is analytic on $\C \setminus \overline{I}_n$ so $r_n^\pm: z \in \C^\pm \mapsto r_n(z)$ is well defined. In light of \eqref{l1}, it suffices that each $r_n^\pm$, with $n \in \N^*$, be suitably extended to some locally H\"older continuous function in $\overline{\C^\pm}$, to derive a LAP for the operator $H$.

\subsection{Singular Cauchy integrals}
Let $n \in \N^*$ be fixed. Bearing in mind that $\lambda_{n}$ is an analytic diffeomorphism from $\R$ onto $I_n$, we note $\lambda_n^{-1}$ the function inverse to $\lambda_n$ and put $\tilde{\psi}(\lambda):=(\psi \circ \lambda_{n}^{-1})(\lambda)$ for any function $\psi : \R \to \R$. Then, upon performing the change of variable $\lambda=\lambda_{n}(k)$ in the integral appearing in \eqref{l1}, we get for every $z \in \C \setminus \overline{I_n}$ that
\bel{D:Hlambda}
r_{n}(z)=\int_{I_n}\frac{H_{n}(\lambda)}{\lambda-z}\rd \lambda\ \mbox{with}\ H_{n}:=\frac{\tilde{f}_{n}\ \overline{\tilde{g}_{n}}}{\tilde{\lambda}_{n}'}=\frac{(f_{n}\circ \lambda_{n}^{-1}) \overline{(g_{n}\circ \lambda_{n}^{-1}})}{\lambda_{n}'\circ \lambda_{n}^{-1}}.
\ee
Therefore, the Cauchy integral $r_n(z)$ is singular for $z \in \overline{I}_n=[E_n,+\infty)$. 
Our main tool for extending singular Cauchy integrals of this type to locally H\"older continuous functions in $\overline{\C^\pm}$ is the Plemelj-Privalov Theorem  (see e.g. \cite[Chap. 2, \S 22]{Mus92}), stated below.
\begin{lemma}
\label{L:KP}
Let $\alpha \in (0,1]$ and let $\psi \in \cC^{0,\alpha}(\overline{I})$, where $I:=(a,b)$ is an open bounded subinterval of $\R$.
Then the mapping $z \mapsto r(z):=\int_I \frac{\psi(t)}{t-z} \rd t$, defined in $\overline{\C^\pm} \setminus \overline{I}$, satisfies for every $\lambda \in I$:
$$ \lim_{\varepsilon \downarrow 0} r(\lambda \pm i \varepsilon) = r^\pm(\lambda):=\mbox{p.v.} \left( \int_I \frac{\psi(t)}{t-\lambda} \rd t \right) \pm i \pi \psi(\lambda). $$
Moreover the function
$$ r^\pm(\lambda):= \left\{ \begin{array}{cl} r(\lambda) & \mbox{if}\ \lambda \in \overline{\C^\pm} \setminus \overline{I} \\ r^{\pm}(\lambda) & \mbox{if}\ \lambda \in I, \end{array} \right. $$
is analytic in $\C^\pm$ and locally H\"older continuous of order $\alpha$ in $\overline{\C^\pm} \setminus \{ a , b \}$ in the sense that
there exists $C_{I,\alpha} \in \cC^0((\overline{\C^\pm} \setminus \{ a , b \})^2;\R_+)$ such that
$$ \forall z, z' \in \overline{\C^\pm} \setminus \{ a , b \},\ | r^\pm(z') - r^\pm(z) | \leq \| \psi \|_{C^{0,\alpha}(\overline{I})} C_{I,\alpha}(z,z') | z'-z |^{\alpha}. $$
In addition, if $\psi(a)=\psi(b)=0$, then $r^{\pm}$ extends to a locally H\"older continuous function of order $\alpha$ in $\overline{\C^\pm}$.
\end{lemma}

\subsection{Limiting absorption principle outside the thresholds}
\label{SS:farseuil}
In this subsection we establish a LAP for $H$ outside its thresholds $\cT=\{E_{n},\ n \in \N^* \}$. This is a rather standard result that we state here for the convenience of the reader. For the sake of completeness we also recall its proof, which requires several ingredients that are useful in the derivation of the main result of subsection
\ref{sec-lap}.
 
  
\begin{proposition}
\label{P:farseuil}
Let $K$ be a compact subset of $\C \setminus \cT$. Then for all $s \in (1 \slash 2,+\infty)$ and any $\alpha \in [0,\min(1,s-1 \slash 2))$, the resolvent $z \mapsto R^{\pm}(z)$ extends to $K \cap \overline{\C^\pm}$ in a H\"older continuous function of order $\alpha$, still denoted by $R^\pm$, for the topology of the norm in $\cB(L^{2,s}(\Omega),L^{2,-s}(\Omega))$. Namely there exists a constant $C=C(K,s,\alpha)>0$, such that the estimate
$$  \| (R^\pm(z') - R^\pm(z)) f \|_{L^{2,-s}(\Omega)} \leq C \| f \|_{L^{2,s}(\Omega)} | z' - z |^{\alpha} $$
holds for all $f \in L^{2,s}(\Omega)$ and all $z , z' \in K \cap \overline{\C^\pm}$.
\end{proposition}
\begin{proof}
Let $f$ and $g$ be in $L^{2,s}(\Omega)$. The notations below refer to \eqref{l1}--\eqref{D:Hlambda}. \Bk
 Since $K$ is bounded there is necessarily $N=N(K) \in \N^*$ such that $d_K:=\inf_{z \in K} ( E_N -\Re(z) )>0$. This and \eqref{s1} entail through straightforward computations that $z \mapsto \sum_{m \geq N} r_{m}(z)$ is Lipschitz continuous in $K$, with
\bel{E:controleanalytique}
\forall z \in K,\ \left\| \sum_{m \geq N} r_{m}(z) \right\|_{\cC^{0,1}(K)} \leq d_K^{-2} \|f\|_{L^2(\Omega)} \|g\|_{L^2(\Omega)}.
\ee
Thus it suffices to examine each $r_{m}$, for $m=1,\ldots, N-1$, separately. Using that $K$ is a compact subset of 
$\C \setminus \cT$ we pick an open bounded subinterval $I=(a,b)$, with $E_m<a<b$, such that $K \cap I_m \subset I$. With reference to \eqref{D:Hlambda} we have the following decomposition for each $z \in K \cap \C^\pm$, 
\bel{l2}
r_m(z) = r_{m}(z;I) + r_m(z;I_m \setminus \overline{I})\ \mbox{where}\ r_{m}(z,J) := \int_J \frac{H_m(z)}{\lambda-z} \rd \lambda\ \mbox{for any}\ J \subset I_m. 
\ee
 Since $E_{m}\notin \overline{I}$, $\lambda_m^{-1}$ and $\lambda_m' \circ \lambda_m^{-1}$ are both Lipschitz continuous in 
$\overline{I}$, and $\inf_{\lambda \in \overline{I}} | \lambda_m' \circ \lambda_m^{-1}(\lambda) | >0$. Therefore we deduce from Lemma \ref{L:fnholder} that $H_m \in \cC^{0,\alpha}(\overline{I})$ and verifies
\bel{l3}
\| H_m \|_{\cC^{0,\alpha}(\overline{I})} \leq c_m \| f \|_{L^{2,s}(\Omega)} \| g \|_{L^{2,s}(\Omega)},
\ee
for some constant $c_m>0$ which is independent of $f$ and $g$. From this and Lemma \ref{L:KP} then follows that $r_m(\cdot;I)$ extends to a locally H\"older continuous function of order $\alpha$ in $K \cap \overline{\C^\pm}$, satisfying
\bel{l4}
\| r_m(\cdot;I) \|_{\cC^{0,\alpha}(K \cap \overline{\C^\pm})} \leq  C_m \| H_m \|_{\cC^{0,\alpha}(\overline{I})},
\ee
where the positive constant $C_m$ depends neither on $f$ nor on $g$. 

Next, as 
the Euclidean distance between $I_m \setminus \overline{I}$ and $K \cap I_m$ is positive, from the very definition of $I$, we get that $\delta_K(I):=\inf \{ | \lambda - z |,\ \lambda \in I_m \setminus \overline{I},\ z \in K \cap \overline{\C^{\pm}} \} > 0$, since $K$ is compact. Therefore $z \mapsto r_m(z;I_m \setminus \overline{I})$ is Lipschitz continuous in $K$ and satisfies
\bel{l5} 
\| r_m(\cdot;I_m \setminus \overline{I}) \|_{\cC^{0,1}(K)} \leq  \delta_K(I)^{-2} \| f \|_{L^{2}(\Omega)} \| g \|_{L^{2}(\Omega)},
\ee
according to \eqref{s1}.

Finally, putting \eqref{E:controleanalytique}-\eqref{l2} and \eqref{l3}-\eqref{l5} together, and recalling that the injection $L^{2,s}(\Omega) \hookrightarrow L^{2}(\Omega)$ is continuous, we end up getting a constant $C>0$, such that the estimate
$$ 
|\langle R^\pm(z) f, g \rangle_{L^2(\Omega)} - \langle R^\pm(z')f,g \rangle_{L^2(\Omega)} |  \leq C \| f\|_{L^{2,s}(\Omega)} \|g\|_{L^{2,s}(\Omega)} |z-z'|^{\alpha}, 
$$
holds uniformly in $f , g \in L^{2,s}(\Omega)$ and $z,z'\in K \cap \overline{\C^{\pm}}$. The result follows from this and the fact that $L^{2,-s}(\Omega)$ and the space $\cB(L^{2,s}(\Omega),\C)$ of continuous linear forms on $L^{2,s}(\Omega)$ are isometric, with the duality pairing
$$ \forall f \in L^{2,-s}(\Omega),\ \forall g \in L^{2,s}(\Omega),\ \langle f , g \rangle_{L^{2,-s}(\Omega), L^{2,s}(\Omega)} := \int_{\Omega} f(x,y) \overline{g(x,y)} \rd x \rd y. $$
\end{proof}

\subsection{Limiting absorption principle at the thresholds}
\label{sec-lap}
We now examine the case of a compact subset $K \subset \C$ containing one or several thresholds of $H$, i.e. such that $K \cap \cT \neq \emptyset$. Since $\lim_{n \to +\infty} E_n=+\infty$ then the bounded set $K$ contains at most a finite number of thresholds. For the sake of clarity, we first investigate the case where $K$ contains exactly one threshold:
\bel{t1} 
\exists n \in \N^*,\ K \cap \cT = \{ E_n \}.
\ee
The target here is the same as in subsection \ref{SS:farseuil}, that is to establish a LAP for $H$ in $K$. Actually, for all $s \in (1 \slash 2,+\infty)$ and any $\alpha \in [0,\min(s-1 \slash 2, 1))$, it is clear from the proof of Proposition \ref{P:farseuil} that $z \mapsto \sum_{m \neq n} r_m(z)$ can be regarded as a $\alpha$-H\"older continuous function in
$K \cap \overline{\C^\pm}$ with values in $\cB(L^{2,s}(\Omega),L^{2,-s}(\Omega))$. 

Thus we are left with the task of suitably extending $z \mapsto r_n(z)$ in $K \cap \overline{\C^\pm}$. But, as $H_n(\lambda)$ may actually blow up as $\lambda$ tends to $E_n$, obviously the method used in the proof of Proposition \ref{P:farseuil} does not apply to $r_n(z)$ when $z$ lies in the vicinity of $E_n$. This is due to the vanishing of the denominator $\lambda_n' \circ \lambda_n^{-1}(\lambda)$ of $H_n(\lambda)$ in \eqref{D:Hlambda} as $\lambda$ approaches $E_n$, or, equivalently, to the flattening of $\lambda_n'(k)$ when $k$ goes to $+\infty$. We shall compensate this peculiar asymptotic behavior of the dispersion curve $k \mapsto \lambda_n(k)$ by imposing appropriate conditions on the functions $f$ and $g$ so the numerator $k \mapsto f_n(k) \overline{g_n(k)}$ decays sufficiently fast at $+\infty$.
This require that the following useful functional spaces be preliminarily introduced.

\paragraph{Suitable functional spaces.}
For any open subset $J \subset I_{n}=(E_{n},+\infty)$ and the non vanishing function 
\begin{equation}
\label{defmun}
\mu_{n}:=|\lambda_{n}'\circ \lambda_{n}^{-1}|^{-1/2},
\end{equation} 
on $J$, we denote by
$\cC^{0,\alpha}_{\mu_{n}}(\overline{J}):=\{ \psi,\  \mu_{n} \psi \in \cC^{0,\alpha}(\overline{J}) \}$\footnote{Since $\mu_n$ is not defined at $E_n$ then it is understood in the peculiar case where $E_n \in \overline{J}$ that $\psi \in \cC^{0,\alpha}_{\mu_{n}}(\overline{J})$ if and only if $\mu_n \psi$ extends continuously to a function lying in $C^{0,\alpha}(\overline{J})$.}(resp. $L_{\mu_n}^2(J):=\{ \psi,\ \mu_n \psi \in L^2(J) \}$), 
the $\mu_n$-weighted space of H\"older continuous functions of order $\alpha \in [0,1)$ (resp., square integrable functions) in $J$. Endowed
with the norm $\|\psi\|_{\cC^{0,\alpha}_{\mu_{n}}(\overline{J})}:=\|\mu_{n}\psi\|_{\cC^{0,\alpha}(\overline{J})}$ (resp., $\| \psi \|_{L_{\mu_n}^2(J)}:= \| \mu_n \psi \|_{L^2(J)}$), $\cC^{0,\alpha}_{\mu_{n}}(\overline{J})$ (resp., $L_{\mu_n}^2(J)$) is a Banach space since this is the case for $\cC^{0,\alpha}(\overline{J})$ (resp., $L^2(J)$). 

Further, the above definitions translate through the linear isometry $\Lambda_{n} : f \mapsto f \circ \lambda_{n}^{-1}$ from $L^2(\R)$ into $L_{\mu_n}^2(I_{n})$, to
$$ \cK^{\alpha}_{n}(\R) =\Lambda_{n}^{-1}(\cC^{0,\alpha}_{\mu_{n}}(\overline{I}_n) \cap L_{\mu_n}^2(I_n)) := \{ f \in L^2(\R),\ \Lambda_n f \in \cC^{0,\alpha}_{\mu_{n}}(\overline{I}_n) \}, $$
which is evidently a Banach space for the norm $\| f \|_{\cK^{\alpha}_{n}(\R)}:=\| \Lambda_n f \|_{\cC^{0,\alpha}_{\mu_{n}}(\overline{I}_n)} + \| \Lambda_n f \|_{L_{\mu_n}^2(I_n)}$.
As a consequence the set
$$ \cX_{n}^{s,\alpha}(\Omega)=\pi_{n}^{-1}(\cK^{\alpha}_{n}(\R))\cap L^{2,s}(\Omega) := \{ f \in L^{2,s}(\Omega),\ \pi_n f \in \cK^{\alpha}_{n}(\R) \},\ s \in \R_+^*, $$
equipped with its natural norm $\|f\|_{\cX_{n}^{s,\alpha}(\Omega)}:=\| \pi_{n}f \|_{\cK^{\alpha}_{n}(\R)}+\|f\|_{L^{2,s}(\Omega)}$ is a Banach space as well.

On $\cC^{0,\alpha}_{\mu_{n}}(\overline{I}_n)$ we define the linear form $\delta_{E_n} : \psi \mapsto (\mu_{n}\psi)(E_{n})$. Notice from the embedding
$\cC^{0,\alpha}_{\mu_{n}}(\overline{I}_n) \subset \cC^{0}_{\mu_n}(\overline{I}_n) := \{ \psi,\ \mu_n \psi \in \cC^0(\overline{I}_n) \}$,
that $\delta_{E_n}$ is well defined since $I_n \ni \lambda \mapsto (\mu_n f)(\lambda)$ extends to a continuous function in $\overline{I}_n$. 
Furthermore, we have $|\delta_{E_n}(\psi)| \leq \| \psi \|_{\cC^{0,\alpha}_{\mu_{n}}(\overline{I}_n)}$ for any $\psi \in \cC^{0,\alpha}_{\mu_{n}}(\overline{I}_n)$ so the linear form $\delta_{E_n}$ is continuous on $\cC^{0,\alpha}_{\mu_{n}}(\overline{I}_n)$. Let us now introduce the subspace
\begin{align*}
\cX_{n,0}^{s,\alpha}(\Omega):=\ &\cX_{n}^{s,\alpha}(\Omega) \cap (\Lambda_{n}\pi_{n})^{-1}(\ker \delta_{E_n}) 
\\
= \ & \{ f \in L^{2,s}(\Omega),\ \mu_n \tilde{f}_n \in \cC^{0,\alpha}(\overline{I}_n) \cap L^2(I_n)\ \mbox{and}\ (\mu_n \tilde{f}_n)(E_n)=0 \},
 \end{align*}
where, as usual, $\tilde{f}_{n}$ stands  for $(\Lambda_{n}\pi_{n})f$. Since $\delta_{E_n}$ is continuous then $\cX_{n,0}^{s,\alpha}(\Omega)$ is closed in $\cX_{n}^{s,\alpha}(\Omega)$. Therefore it is a Banach space for the norm 
$$\| f \|_{\cX_{n}^{s,\alpha}(\Omega)}=\| \tilde{f}_n \|_{\cC_{\mu_n}^{0,\alpha}(\overline{I}_n)} + \| \tilde{f}_n \|_{L_{\mu_n}^2(I_n)} + \|f\|_{L^{2,s}(\Omega)}.$$
Moreover, $\cC^{\infty}_{0}(\R)$ being dense in $L^2(\R)$, we deduce from the imbedding $\pi_{n}^{-1}(\cC^{\infty}_{0}(\R))\subset \cX_{n,0}^{s,\alpha}(\Omega)$ that $\cX_{n,0}^{s,\alpha}(\Omega)$ is dense in $L^2(\R^2)$ for the usual norm-topology.

Summing up, we have obtained the:
\begin{lemma}
\label{lm-density}
The set $\cX_{n,0}^{s,\alpha}(\Omega)$ is a Banach space and is dense in $L^2(\Omega)$.
\end{lemma}

\paragraph{\it Absorption at $E_n$.} Having defined $\cX_{n,0}^{s,\alpha}(\Omega)$ for $n \in \N^*$ fixed, we now derive a LAP at $E_n$ for the restriction of the operator $H$ to $\cX_{n,0}^{s,\alpha}(\Omega)$ associated with suitable values of $s$ and $\alpha$.
 
\begin{proposition}
\label{P:onethreshold}
Let $K$ be a compact subset of $\C$ obeying \eqref{t1} and let $s \in (1 \slash 2, +\infty)$. Then, for every $\lambda \in K \cap \R$ and $\alpha \in [0,\min(1,s-1 \slash 2))$
, both limits $\lim_{z \to \lambda,\ \pm \Im \lambda >0} R(z)$ exist in the uniform operator topology on $\cB(\cX_{n,0}^{s,\alpha}(\Omega),(\cX_{n,0}^{s,\alpha}(\Omega))')$ . Moreover the resolvent $R^\pm$ extends to a H\"older continuous on $K \cap \overline{\C^\pm}$ with order $\alpha$; Namely there exists $C>0$ such that we have
$$ \forall z,z'\in K \cap\overline{\C^\pm},\ \forall f \in \cX_{0,n}^{s,\alpha}(\Omega),\ \|  ( R^\pm(z) - R^\pm(z') ) f \|_{(\cX_{0,n}^{s,\alpha}(\Omega))'} \leq C |z-z'|^{s,\alpha} \|f\|_{\cX_{n}^{s,\alpha}(\Omega)}. $$
 \end{proposition}
\begin{proof}
It is clear from \eqref{t1} upon mimicking the proof of Proposition \ref{P:farseuil}, that $z \mapsto \sum_{m \neq n} r_m(z)$ extends to an $\alpha$-H\"older continuous function in $K \cap \overline{\C^\pm}$, denoted by $\sum_{m \neq n} r_m^{\pm}$, satisfying
\bel{t10}
\left\|  \sum_{m \neq n} r_m^{\pm} \right\|_{C^{0,\alpha}(K \cap \overline{\C^\pm})} \leq c \| f \|_{L^{2,s}(\Omega)} \| g \|_{L^{2,s}(\Omega)},
\ee
for some constant $c>0$ that depends only on $K$, $s$ and $\alpha$. 

We turn now to examining $r_n$. Taking into account that $K$ is bounded we pick $\epsilon >0$ so large that $K \cap I_n \subset J:=(E_n,E_n+\epsilon)$ and refer once more to the proof of Proposition \ref{P:farseuil}. We get that $z \mapsto r_n(z;I_n \setminus \overline{J})$ extends to a H\"older continuous function $r_n^{\pm}(\cdot;I_n \setminus \overline{J})$ 
of exponent $\alpha$ in $K \cap \overline{\C^{\pm}}$, with 
\bel{t11}
\left\| r_n^{\pm}(\cdot; I_n \setminus \overline{J} ) \right\|_{C^{0,\alpha}(K \cap \overline{\C^\pm})} \leq c' \| f \|_{L^{2,s}(\Omega)} \| g \|_{L^{2,s}(\Omega)},
\ee
where $c'>0$ is a constant depending only on $K$, $s$ and $\alpha$ and $\epsilon$. 

Finally, since $f$ and $g$ are taken in $\cX_{n,0}^{s,\alpha}(\Omega)$ then the function $\lambda \mapsto H_{n}(\lambda)$, defined in \eqref{D:Hlambda}, is $\alpha$-H\"older continuous in $J$, and we have
\bel{t11b} 
\| H_{n}\|_{\cC^{0,\alpha}(\overline{J})} \leq \| \Lambda_{n} f_{n} \|_{\cC_{\mu_{n}}^{0,\alpha}(\overline{J})} \|\Lambda_{n} g_{n} \|_{\cC_{\mu_{n}}^{0,\alpha}(\overline{J})} 
\leq \|f_{n}\|_{\cK_{n}^{\alpha}(\R)}\|g_{n}\|_{\cK_{n}^{\alpha}(\R)} \leq \|f\|_{\cX_{n}^{\alpha}(\Omega)} \|g\|_{\cX_{n}^{\alpha}(\Omega)}.
\ee
Bearing in mind that $H_n(E_n)=0$ and $E_n+ \epsilon \notin K$, we deduce from \eqref{t11b} and Lemma \ref{L:KP} that $z \mapsto r_{n}^{\pm}(\cdot;J)$ extends to an $\alpha$-H\"older continuous function, still denoted by $r_{n}^\pm(\cdot;J)$, in $K \cap \overline{\C^\pm}$, obeying
\bel{t12}
\| r_{n}^\pm(\cdot;J) \|_{\cC^{0,\alpha}(K \cap \overline{\C^\pm})} \leq c' \| H_{n}\|_{\cC^{0,\alpha}(\overline{J})} \leq c' \|f\|_{\cX_{n}^{s,\alpha}(\Omega)} \|g\|_{\cX_{n}^{s,\alpha}(\Omega)},
\ee
where $c'$ is the same as in \eqref{t11}.
Finally, putting \eqref{t10}-\eqref{t11} and \eqref{t12} together, we end up getting a constant $C>0$, which is independent of $f$ and $g$, such that we have
$$ \forall z,z'\in K \cap\overline{\C^\pm},\ | \langle R^\pm(z) f , g \rangle_{L^2(\Omega)} - \langle R^\pm(z') f, g \rangle_{L^2(\Omega)}  | \leq C |z-z'|^{\alpha} \|f\|_{\cX_{n}^{s,\alpha}(\Omega)} \|g\|_{\cX_{n}^{s,\alpha}(\Omega)}. $$
Here we used the basic identity $r_n=r_n(\cdot;I_n \setminus \overline{J}) + r_n(\cdot;J)$ and the continuity of
embedding $\cX_{n,0}^{s,\alpha}(\Omega)\hookrightarrow L^{2,s}(\Omega)$. This entails the desired result.
\end{proof}

For any compact subset $K \subset \C$, the set $\gJ_{K}:=\{m \in \N^*,\ E_m \in K \}$ is finite. Then upon substituting $\cX_{K}^{s,\alpha}(\Omega):=\cap_{m \in \gJ_{K}} \cX_{m,0}^{s,\alpha}(\Omega)$
for $\cX_{n,0}^{s,\alpha}(\Omega)$ in the proof of Proposition \ref{P:onethreshold}, it is apparent that we obtain the:
\begin{theorem} 
\label{T:thresholds}
Let $K \subset \C$ be compact, and let $s$ and $\alpha$ be the same as in Proposition \ref{P:onethreshold}. Then the resolvent $z \mapsto R^\pm(z)$, initially defined on $K \cap \C^\pm$, extends to an $\alpha$-H\"older continuous function on $K \cap \overline{\C^\pm}$ in the uniform operator topology in $\cB(\cX_{K}^{s,\alpha},(\cX_{K}^{s,\alpha})')$.
\end{theorem}

\section{Analytic and decay properties in absorption spaces}
In this section we investigate the analytic and decay properties of functions in $\cX_{n,0}^{s,\alpha}(\Omega)$ for $n \in \N^*$, $s \in (1 \slash 2,+\infty)$ and $\alpha \in [0,\min(1,s-1 \slash 2))$.
We preliminarily establish with the aid of the (explicit) asymptotic behavior of $\lambda_{n}(k)$ as $k \to +\infty$, that their $n$-th Fourier coefficient decays super-exponentially fast. This has two main consequences for the $n$-th harmonic of any function in $\cX_{n,0}^{s,\alpha}(\Omega)$. First, its $L^2(\R^+)$-expectation is analytically extendable to $\overline{\C}^\pm$. Second, and more surprisingly, the $n$-th harmonic's mean value is a super-exponentially decaying function of the distance to the edge.

As already seen in the derivation of Lemma \ref{lm-density}, any $f \in L^{2,s}(\Omega)$ with $s \in (1 \slash 2,+\infty)$, such that $f_{n}$ is lying in $\cC_{0}^{\infty}(\R)$ for some 
$n \in \N^*$, belongs to $\cX_{n,0}^{s,\alpha}(\Omega)$ for every $\alpha \in [0,\min(1,s-1 \slash 2))$.

Conversely, if $f \in \cX_{n,0}^{s,\alpha}(\Omega)$ then $\mu_n \tilde{f}_{n} \in \cC^{0,\alpha}(\overline{I}_{n})$ satisfies $(\mu_{n} \tilde{f}_{n})(E_{n})=0$ so we may find 
a constant $C>0$ such that we have
$|(\mu_{n} \tilde{f}_{n}) (\lambda)| \leq C |\lambda-E_{n}|^{\alpha}$ for every $\lambda \in [E_{n},+\infty)$.
Recalling \eqref{defmun} and performing the change of variable $\lambda=\lambda_{n}(k)$, this may be equivalently reformulated as
\bel{E:majoreponctpn}
\forall k\in \R,\ |w_{n}^{\alpha}(k)f_{n}(k)| \leq C,
\ee
where 
\bel{wna}
w_{n}^{\alpha}(k):=|\lambda_{n}(k)-E_{n}|^{-\alpha}|\lambda_{n}'(k)|^{-1/2}.
\ee
Summing up, we have proved the following:
 
\begin{proposition}
\label{Crit1}
For all $n \in \N^*$, $s \in (1 \slash 2,+\infty)$ and $\alpha \in [0,\min(1,s-1 \slash 2))$, we have the implication:
$$ \left( f \in \cX_{n,0}^{s,\alpha}(\Omega) \right) \Rightarrow \left( w_{n}^{\alpha} f_{n}\in L^{\infty}(\R) \right). $$
\end{proposition}
\begin{remark}
Given $N \in \N^*$ and a suitable weight function $\mu : \R^N \mapsto \R$, the continuous Besov space (see e.g. \cite[\S 10.1]{Hor83}) of order $p >0$ is defined as
$B_{\mu,p} := \{ f \in L^2(\R^N),\ \mu \hat{f} \in L^{p}(\R^N) \}$, where $\hat{f}$ stands for the Fourier transform of $f$.
The occurrence in Proposition \ref{Crit1} of functions $f$ satisfying $w_{n}^{\alpha} f_{n}\in L^{\infty}(\R)$ is reminiscent of $B_{\mu_n^{\alpha},\infty}$ but it turns out that these two sets do not coincide here, since the condition $w_{n}^{\alpha} f_{n}\in L^{\infty}(\R)$ is imposed solely on the $n$-th Fourier coefficient $f_n$ of $f$, and not on the whole Fourier transform (with respect to $y$) $\cF_y f$, of $f$. Moreover there is another major difference with the analysis carried out in \cite[\S 10.1]{Hor83}since the weight function $w_{n}^{\alpha}$ is not a tempered function.
\end{remark}
 
For every $n \in \N^*$, let us now recall from \cite[Theorem 1.4]{HisPofSoc14} the asymptotic behavior as $k\to+\infty$ of the two functions
  \begin{equation*}
  \left\{
  \begin{aligned}
  \lambda_{n}(k)&=E_{n}+C_{n}k^{2n-1}e^{-k^2}(1+O(k^{-2}))
  \\
  \lambda_{n}'(k)&=-2C_{n}k^{2n}e^{-k^2}(1+O(k^{-2}))
  \end{aligned}
  \right.
  \end{equation*}
where $C_{n}:=2^{n}((n-1)!\sqrt{\pi})^{-1}$. In light of \eqref{wna}, this entails for any $\alpha \in [0,1)$,
\bel{A:pnalpha}
 w_{n}^{\alpha}(k) \underset{k\to+\infty}{=}C_{n,\alpha}k^{-n(2\alpha+1)+\alpha}e^{k^2(\alpha+1/2)}(1+O(k^{-2})) \, .
\ee
with $C_{n,\alpha}:=C_{n}^{-\alpha-1/2}2^{-1/2}$. Therefore, with reference to Proposition \ref{Crit1}, we deduce from \eqref{A:pnalpha} that the $n$-th Fourier coefficient $f_n$ of $f \in \cX_{0,n}^{s,\alpha}(\Omega)$ associated with suitable coefficients $s \in (1 \slash 2,+\infty)$ and $\alpha \in [0,\min(1,s-1 \slash 2))$, is a super-exponentially decaying function of the variable $k$. This has several interesting consequences we shall make precise below.

\subsection{Analyticity}
It is well known (see e.g. \cite{ReSi78}) that suitable decay properties of the Fourier transform $\hat{f}$ of a real analytic function $f$ translate into analytic continuation of $f$ to subsets of $\C$. We shall see that \eqref{E:majoreponctpn}--\eqref{A:pnalpha} yield a similar result.
 
\begin{corollary} 
\label{cor-c1}
Let $n$, $s$ and $\alpha$ be the same as in Proposition \ref{Crit1}. Then for any $f\in \cX_{n,0}^{\alpha}(\Omega)$, the function 
$y \mapsto \|\Pi_{n}f(\cdot,y)\|_{L^2(\R_{+})}$, initially defined in $\R$, admits an analytic continuation to $\overline{\C^{-}}$. 
\end{corollary}
\begin{proof}
With reference to \eqref{D:Pn}, we have
\bel{cp1}
\|\Pi_{n}f(\cdot,y)\|_{L^2(\R_{+})}^2=\int_{\R^2} e^{iky}f_{n}(k)\overline{e^{ik'y}f_{n}(k')}F(k,k')\rd k \rd k',
\ee  
for any $y\in \C$, where $F(k,k'):=\int_{\R_{+}}u_{n}(x,k)u_{n}(x,k')\rd x$. Further, taking into account that $|F(k,k')|\leq 1$ for all $(k,k') \in \R^2$ since
$x \mapsto u_{n}(x,k)$ is $L^2(\R_+)$-normalized, we get that
$$ |e^{iky}f_{n}(k)\overline{e^{ik'y}f_{n}(k')}F(k,k')| \leq  e^{-\Im(y)(k+k')}|f_{n}(k)||f_{n}(k')|.$$
Now the result follows from this, \eqref{E:majoreponctpn}-\eqref{A:pnalpha} and \eqref{cp1}, upon applying the dominated convergence theorem (see e.g. \cite[Section IX.3]{ReSi78}).
\end{proof}
\begin{remark}
It is not clear whether the result of Corollary \ref{cor-c1} remains valid for $y\mapsto \Pi_{n}f(x,y)$, uniformly in $x \in \R_+^*$.
This would require that $\|u_{n}(\cdot,k)\|_{L^{\infty}(\R_{+})}$ be appropriately bounded with respect to $k \in \R$, which does not seem to be the case (it expected that this quantity behaves like $\lambda_{n}(k)^{1 \slash 4}$ as $k \to -\infty$, see the proof of Theorem \ref{thm1} for more details).
 \end{remark}
 
\subsection{Geometric localization}
\label{SS:geomloc}
In this section we investigate the geometric properties of functions in $\cX_{0,n}^{s,\alpha}(\Omega)$. Namely we establish that the mean value of the $n$-th harmonic in $(x,+\infty) \times \R$ decays super-exponentially fast with $x>0$, provided $x$ is sufficiently large.

\begin{theorem}
\label{thm1}
Let $n$, $s$ and $\alpha$ be the same as in Proposition \ref{Crit1}. Then for any positive real number $\beta < \min \left( 1,\frac{2\alpha+1}{1+\sqrt{2\alpha+1}} \right)$, we may find two constants $C_n(\beta)>0$ and $L_{n}(\beta)>0$ such that
\bel{g0}
\forall f \in \cX_{0,n}^{s,\alpha}(\Omega),\ \forall L \geq L_n(\beta),\ \int_{L}^{+\infty} \| \Pi_{n}f(x,\cdot) \|_{L^2(\R)}^2 \rd x \leq C_n(\beta) \| f \|_{L^2(\Omega)}^2 e^{- \beta L^2}.
\ee
\end{theorem}

We split the proof of Theorem \ref{thm1} into two parts. With reference to \eqref{D:Pn} we have
\bel{g1} 
\|\Pi_{n}f(x,\cdot)\|_{L^2(\R)}^2=\int_{\R} u_{n}(x,k)^2|f_{n}(k)|^2 \rd k,
\ee
for all $x \in \R_+^*$, by Plancherel theorem, and we examine $\int_{\R_\pm} u_{n}(x,k)^2|f_{n}(k)|^2 \rd k$ separately.
Indeed, the claim of Theorem \ref{thm1} being reminiscent of the super-exponential decay exhibited by the eigenfunctions of Sturm-Liouville operators with a non vanishing potential, such as $\gh(k)$ for $k \in \R_-$ (see e.g. \cite{Iwa85}), it will come as no surprise that the decay property of $\int_{\R_-} u_{n}(x,k)^2|f_{n}(k)|^2 \rd k$ is obtained from the one of $u_n(\cdot,k)$ for $k \in \R_-$. This is the content of Lemma \ref{L:estimekneg}.
However, such a property being no longer valid for $k \in \R_+$, we treat $\int_{\R_+} u_{n}(x,k)^2|f_{n}(k)|^2 \rd k$
in subsection \ref{subsec-completion}, with the aid of Proposition \ref{Crit1}.

\subsubsection{Super-exponential decay of the eigenfunctions and consequences.}

\begin{lemma}
\label{L:estimekneg}
Let $f\in L^2(\Omega)$ and $n \geq 1$. Then for any $\beta \in (0,1)$, we may find two constants $c_n(\beta)>0$ and $\ell_n(\beta)>0$, depending only on $n$ and $\beta$, such that the estimate
$$ \int_{(L,+\infty)\times \R_-} u_{n}(x,k)^2 |f_{n}(k)|^2 \rd x \rd k \leq c_{n}(\beta) \| f \|_{L^2(\Omega)}^2 e^{-\beta L^2}, $$
holds uniformly for $L \in [\ell_n(\beta),+\infty)$.
\end{lemma}
\begin{proof}
We fix $k \in \R_-$ and combine the Feynman-Hellmann formula with the Cauchy-Schwarz inequality, getting
$$
- \lambda_n'(k) =  2 \int_{\R_+} (x-k) u_n(x,k)^2 \rd x \leq  2 \| (\cdot-k) u_n(\cdot,k) \|_{L^2(\R_+)},
$$
since $ \| u_n(\cdot,k) \|_{L^2(\R_+)}=1$. Thus, recalling from the energy estimate 
\bel{g4a}
\| u_n'(\cdot,k) \|_{L^2(\R_+)}^2 + \| (\cdot-k) u_n(\cdot,k) \|_{L^2(\R_+)}^2 = \lambda_n(k), 
\ee
arising from the eigenvalue equation 
\bel{g4b}
u_n''(x,k) =q_n(x,k) u_n(x,k),\ \mbox{where}\ q_n(x,k):=(x-k)^2 - \lambda_n(k),
\ee
that $\| (\cdot-k) u_n(\cdot,k) \|_{L^2(\R_+)} \leq \lambda_n(k)^{1 \slash 2}$, we find out that
$- \lambda_n'(k) \leq 2 \lambda_n(k)^{1 / 2}$. An integration over $(k,0)$ yields
$\lambda_n(k)^{1 / 2} \leq \lambda_n(0)^{1 / 2} - k$, hence
\bel{g5}
\forall k \in \R_-,\ \lambda_n(k) \leq (x_n - k)^2,\ \mbox{where}\ x_n:=\lambda_n(0)^{1 / 2}=(4n-1)^{1 / 2}, 
\ee
by \cite[Lemma 2.1 (i)]{DeBPu99}.
Further, taking into account that $u_n(0,k)=0$, we derive from \eqref{g4a} and the $L^2(\R_+)$-normalization of $u_n(\cdot,k)$
that
$$ \forall x \in \R_+,\ u_n(x,k)^2 =2 \int_0^x u_n(t,k) u_n'(t,k) dt \leq 2 \| u_n'(\cdot,k) \|_{L^2(\R_+)} \leq 2 \lambda_n(k)^{1 \slash 2}. $$
From this and \eqref{g5} then follows that
\bel{g6}
\forall k \in \R_-,\ \| u_n(\cdot,k) \|_{L^{\infty}(\R_+)} \leq  2^{1 /2} \lambda_n(k)^{1 \slash 4} \leq 2^{1 /2} (x_n - k)^{1 \slash 2}.
\ee
The next step involves noticing from \eqref{g5} that the effective potential
$q_n(x,k)$ is non negative for every $x \in [x_n,+\infty)$, uniformly in $k \in \R_-$:
\bel{g6a}
\forall x \geq x_n,\ q_n(x,k) \geq (x-k)^2-(x_n-k)^2 \geq 0.
\ee
Therefore, the $H^1(\R_+)$-solution $u_n(\cdot,k)$ of the equation \eqref{g4b} satisfies
\bel{g6b}
\lim_{x \to +\infty} q_n(x,k) u_n(x,k)^2=0,
\ee
and
\bel{g6c}
\forall x > x_n,\ u_n(x,k) u_n'(x,k) < 0,
\ee
in virtue of \cite[Proposition B.1]{HiSoc1}. Next, by multiplying \eqref{g4b} by $u_n'(x,k)$, integrating over $(t,v)$ for $t> v \geq x_n$, and sending $v$ to $+\infty$, we deduce from \eqref{g6b} that 
$u_n'(t,k)^2 \leq q_n(t,k) u_n(t,k)^2$ for every $t > x_n$. In view of \eqref{g6c}, this entails that
\bel{g7}
\forall x \geq x_n,\ | u_n(x,k) | \leq | u_n(x_n,k) | e^{-\int_{x_n}^x q_n(t,k)^{1 \slash 2} \rd t}.
\ee
In light of \eqref{g6a}, it holds true that
$$
\forall x \geq x_n,\ q_n(x,k)^{1 \slash 2} \geq \beta (x-x_n) + (1-\beta^2)^{1 \slash 2} (x-x_n)^{1 \slash 2} (x_n-k)^{1 \slash 2},
$$
as we have $a^2+b^2 \geq (\beta a + (1-\beta^2)^{1 \slash 2} b)^2$ for any real numbers $a$ and $b$.
Putting this together with \eqref{g6} and \eqref{g7}, we get that
$$ 
\forall x \geq x_n,\ \forall k \in \R_-,\ |  u_n(x,k) | \leq 2^{1 /2} (x_n - k)^{1 / 2} e^{\frac{2}{3}(1-\beta^2)^{1 \slash 2}( x_n - k )^{1 / 2} (x-x_n)^{3 / 2}} e^{-\frac{\beta}{2}( x-x_n)^{2}}. $$
From this and the estimates $L-x_n \geq (L \slash 2) \geq x_n$, then follows that
$$
\int_L^{+\infty} u_n(x,k)^2 \rd x \leq  2 (x_n-k) e^{-\frac{4}{3}(1-\beta^2)^{1 \slash 2} x_n^{3 \slash 2} (x_n-k)^{1 \slash 2} } \int_L^{+\infty} e^{-\beta(x-x_n)^2}  \rd x .
$$
Now, taking into account that $2 \beta^{1 \slash 2} \int_L^{+\infty} e^{-\beta(x-x_n)^2}  \rd x \leq \pi^{1 \slash 2} e^{-\beta(L-x_{n})^2}$, we have
$$ \forall L \geq 2 x_n,\ \forall k\in \R,\ \int_L^{+\infty} u_n(x,k)^2 \rd x \leq c_{n}(\beta) e^{-\beta (L-x_n)^2}, $$
where the constant $c_n(\beta):=\sup_{k \leq 0} \left( \frac{\pi}{\beta} \right)^{1 \slash 2} (x_n - k) e^{-\frac{4}{3} (1-\beta^2)^{1 \slash 2} x_n^{3 \slash 2} ( x_n - k )^{1 / 2}} < +\infty$.
Therefore, upon eventually shortening $\beta$, we may choose $\ell_n(\beta) \geq 2 x_n$ so large that the estimate
$$ \int_{(L,+\infty)\times \R_-} u_{n}(x,k)^2|f_{n}(k)|^2 \rd x \rd k \leq c_n(\beta) e^{-\beta L^2} \int_{-\infty}^0 | f(k)|^2 \rd k \leq c_n(\beta)\|f\|_{L^2(\Omega)}^2 e^{-\beta L^2}, $$
holds for every $L \geq \ell_n(\beta)$. This terminates the proof.
\end{proof}

The result of Lemma \ref{L:estimekneg} is no longer valid for $\Pi_n^+ f$, as the effective potential $q_n(\cdot,k)$ of $\gh(k)$ vanishes in $\R_+^*$, for $k \in \R_+$.
Nevertheless we shall see that the geometric localization of $x \mapsto \| \Pi_n^+ f(x,\cdot) \|_{L^2(\R)}$ relies on the decay properties \eqref{E:majoreponctpn}--\eqref{A:pnalpha} of the Fourier coefficient $f_n$, arising from the assumption $f \in \cX_{0,n}^{s,\alpha}(\Omega)$.

\subsubsection{Completion of the proof of Theorem \ref{thm1}}
\label{subsec-completion}

Then for $L>0$ fixed, we introduce $k(L)>0$ (we shall make precise further) such that $k(L) \to +\infty$ as $L\to+\infty$, and deduce from \eqref{g1} that
\bel{E:decoupagek}
\int_{L}^{+\infty}  \| \Pi_{n}(x,\cdot) \|_{L^2(\R)}^2 \rd x = \sum_{j=0}^2 \mathcal{I}_j,
\ee
where $\mathcal{I}_j := \int_{(L,+\infty) \times I_j} u_{n}(x,k)^2 |f_{n}(k)|^2 \rd x \rd k$ for $j=1,2$, $I_0:=(-\infty,0)$, $I_1:=(0,k(L))$ and $I_2:=(k(L),+\infty)$. 
$\mathcal{I}_0$ is already estimated in Lemma \ref{L:estimekneg}. We treat each $\mathcal{I}_j$, $j=1,2$, separately. We start with $j=2$.
Actually, it is clear from \eqref{E:majoreponctpn} and \eqref{A:pnalpha} that for $L$ sufficiently large, we have
\bel{g2}
\mathcal{I}_2 \leq \int_{k(L)}^{+\infty} |f_{n}(k)|^2 \rd k \leq C \int_{k(L)}^{+\infty} k^{2(2\alpha+1)n}e^{-(2\alpha+1) k^2} \rd k
\leq C k(L)^{2(2\alpha+1)n-1}e^{-(2\alpha+1)k(L)^2}.
\ee
Here and henceforth, $C$ denotes some generic positive constant that is independent of $L$.

Next, bearing in mind that the Agmon distance associated with the operator $\gh(k)$ is explicitly known and equals $x \mapsto (x-k)^2 \slash 2$, and that $\lambda_{n}(k) \leq 4n-1$ for every $k \geq 0$ by \cite[Lemma 2.1 (i), (ii)]{DeBPu99}, we get
$$\forall \beta\in (0,1),\exists C>0, \forall (x,k)\in \R_{+}\times \R_{+}, \quad |u_{n}(x,k)| \leq C e^{-\beta\frac{(x-k)^2}{2}}, $$
by applying standard Agmon estimates (see e.g. \cite{Ag85, He88}).
Assuming that $k(L)<L$, we deduce from this for any fixed $\beta \in (0,1)$, that
\bel{g3}
\mathcal{I}_1 \leq C \int_{L}^{+\infty} e^{-\beta(x-k(L))^2} \rd x \int_{0}^{k(L)} |f_{n}(k)|^2 \rd k \leq \frac{C}{L-k(L)} e^{-\beta(L-k(L))^2}.
\ee
Now, taking 
$$
k(L):=\gamma L\ \mbox{where}\ \gamma = \gamma(\alpha,\beta):= \frac{\sqrt{\beta}}{1+\sqrt{2\alpha+1}} \in (0,1), 
$$
in such a way that the arguments of the exponential terms appearing in the upper bound of \eqref{g2}-\eqref{g3} coincide, we obtain
$$ 
\sum_{j=1,2} \mathcal{I}_j \leq C L^{2(2\alpha+1)n-1}e^{- \beta \frac{(2\alpha+1)}{1+\sqrt{2\alpha+1}}L^2}.
$$
This and Lemma \ref{L:estimekneg} then yield that $\sum_{j=0,1,2} \mathcal{I}_j \leq C e^{-\beta \min \left(1,\frac{(2\alpha+1)}{1+\sqrt{2\alpha+1}} \right)L^2}$
for $L$ sufficiently large, so the desired result follows from \eqref{E:decoupagek}.

\begin{remark}
\label{R:symbole}
Notice from subsection \ref{subsec-completion} that the second ingredient we used in the proof of Theorem \ref{thm1} are the decay properties of the Fourier coefficient $f_n$ (in the $k$ variable, that is the Fourier variable associated with $y$) which translate into exponential decay for $\pi_n^+ f$ in the $x$-variable. This is due to a purely magnetic effect arising from the mixing of the space variable $x$ with the frequency $k$ in phase space, through the symbol of the operator $H$. 
\end{remark}

%

\newpage

\end{document}